\documentclass{amsart}

\usepackage{amsmath}
\usepackage{amssymb}
\usepackage{amsthm}
\usepackage[all]{xy}

\swapnumbers
\theoremstyle{plain}
\newtheorem*{mthmA}{Theorem A}
\newtheorem*{mthmB}{Theorem B}
\newtheorem*{mthmC}{Theorem C}
\newtheorem{lem}{Lemma}[section]
\newtheorem{cor}[lem]{Corollary}
\newtheorem{prop}[lem]{Proposition}

\theoremstyle{definition}
\newtheorem{ex}[lem]{Example}
\newtheorem{rem}[lem]{Remark}
\newtheorem{dfn}[lem]{Definition}

\newcommand{\CH}{\text{\upshape{CH}}}   % Chow ring
\newcommand{\codim}{\text{\upshape{codim }}}  % codimension
\newcommand{\tCH}{\text{\tiny\upshape{CH}}} % upscript Chow

\newcommand{\Rs}{\mathcal{R}}    % indecomposable motive

\newcommand{\cs}{\mathfrak{c}}   % characteristic map

\newcommand{\Sm}{\text{\bf{Sm}}}   % smooth varieties
\newcommand{\SmProj}{\text{\bf{SmProj}}}   % smooth projective varieties

\newcommand{\Ms}{\mathcal{M}}   % motivic category

\newcommand{\Spec}{\text{Spec }}  % Spectrum

\newcommand{\frc}[2]{\raisebox{2pt}{$#1$}\big/\raisebox{-3pt}{$#2$}} % quotient

\DeclareMathOperator{\im}{\mathrm{im}}    % image
\DeclareMathOperator{\hh}{\mathtt{h}}   % oriented cohomology

\title{Motives and Oriented cohomology of a linear algebraic group}

\author{Alexander Neshitov}

\address{Department of Mathematics and Statistics, University of Ottawa, Canada}

\email{anesh094@uottawa.ca}

\begin{document}

\maketitle

\begin{abstract}
For a cellular variety $X$ over a field $k$ of characteristic 0 and an algebraic oriented cohomology theory $\hh$ of Levine-Morel we construct a filtration on the cohomology ring $\hh(X)$ such that the associated graded ring is isomorphic to the Chow ring of $X$. Taking $X$ to be the variety of Borel subgroups of a split semisimple linear algebraic group $G$ over $k$ we apply this filtration to relate the oriented cohomology of $G$ to its Chow ring. As an immediate application we compute the algebraic cobordism ring of a group of type $G_2$ and of some other groups of small ranks, hence, extending several results by Yagita. 

Using this filtration we also establish the following comparison result between Chow motives and $\hh$-motives of generically cellular varieties: any irreducible Chow-motivic decomposition of a generically split variety $Y$ gives rise to a $\hh$-motivic decomposition of $Y$ with the same generating function. Moreover, under some conditions on the coefficient ring of $\hh$ the obtained $\hh$-motivic decomposition will be irreducible. We also prove that if Chow motives of two twisted forms of $Y$ coincide, then their $\hh$-motives coincide as well.
\end{abstract}

\section{Introduction}

We work over the base field $k$ of characteristic $0$. For an algebraic oriented cohomology theory $\hh$ of Levine-Morel \cite{LM} and a cellular variety $X$ of dimension $N$ we construct a filtration 
\[
\hh(X)=\hh^{(0)}(X)\supseteq\hh^{(1)}(X)\supseteq\ldots\supseteq \hh^{(N)}(X)\supseteq 0
\] 
on the cohomology ring such that the associated graded ring 
\[
Gr^*\hh(X)=\bigoplus_{i\ge 0} \hh^{(i)}(X)/\hh^{(i+1)}(X)
\] 
is isomorphic (as a graded ring) to the Chow ring $\CH^*(X,\Lambda)$ of algebraic cycles modulo rational equivalence relation with coefficients in a ring $\Lambda$. We exploit this filtration and isomorphism in two different contexts:

\medskip

First, we consider the (cellular) variety $X=G/B$ of Borel subgroups of a split semisimple linear algebraic group $G$ over $k$. By~\cite[Prop.~5.1]{GZ} the cohomology ring $\hh(G)$ can be identified with a quotient of $\hh(G/B)$, so there is an induced filtration on $\hh(G)$. One of our key results (Prop.~\ref{GroupComparison}) shows that $\CH^*(G,\Lambda)$ covers the associated graded ring $Gr^*\hh(G)$ and describes the kernel of this surjection. As an immediate application for $\hh=\Omega$ (the algebraic cobordism of Levine-Morel) we compute the cobordism ring for groups $G_2$, $SO_3$, $SO_4$, $Spin_n$ for $n=3,4,5,6$ and $PGL_n$ for $n\ge 2$, in terms of generators and relations, hence, extending several previously known results by Yagita \cite{Yagita}; as an application for $\hh=K_0$ (the Grothendieck $K_0$) we construct certain elements in the difference between topological and the Grothendieck $\gamma$-filtration on $K_0(X)$, hence, extending some of the results by Garibaldi-Zainoulline \cite{GarZ}.

\medskip

The second deals with the study of $\hh$-motives of generically cellular varieties. The latter is a natural generalization of the notion of the Chow motives to the case of an arbitrary algebraic oriented cohomology theory of Levine-Morel. It was introduced and studied by Nenashev-Zainoulline in~\cite{NZ} and Vishik-Yagita in~\cite{VY}. 

Let $\Lambda$ denote the coefficient ring of $\hh$ and let $\Lambda^i$ denote its $i$-th graded component.
We prove the following theorem which relates $\hh$-motives of generically cellular varieties to its Chow motives:

\begin{mthmA}
Let $X$ be a generically cellular variety over $k$, i.e. cellular over the function field $k(X)$.
Assume that the Chow motive of $X$ with coefficients in $\Lambda^0$ splits as
\[
M^{\tCH}(X,\Lambda^0)=\bigoplus_{i\ge 0} \Rs(i)^{\oplus c_i},\; c_i \ge 0,
\] 
for some motive $\Rs$ which splits as a direct sum of twisted Tate motives $\overline{\Rs}=\bigoplus_{j\ge 0}\Lambda^0(j)^{\oplus d_j}$ over its splitting field. 

Then the $\hh$-motive of $X$ (with coefficients in $\Lambda$) splits as 
\[
M^{\hh}(X)=\bigoplus_{i\ge 0} \Rs_{\hh}(i)^{\oplus c_i}
\]
for some motive $\Rs_{\hh}$ and over the same splitting field $\Rs_{\hh}$ splits as a direct sum of twisted $\hh$-Tate motives 
$\overline{\Rs_{\hh}}=\bigoplus_{j\ge 0}\Lambda(j)^{\oplus d_j}$.
\end{mthmA}

This result can also be derived from the arguments of~\cite{VY} where it is proved that sets of isomorphism classes of objects of category of Chow motives and $\Omega$-motives coincide. However, our approach gives more explicit correspondence between idempotents defining  the (Chow) motive $\Rs$ and the $\hh$-motive $\Rs_{\hh}$. The latter allows us to prove the following result concerning the indecomposability of the $\hh$-motive $\Rs_{\hh}$:

\begin{mthmB} Assume that $\Lambda^1=\ldots\Lambda^N=0$, where $N=\dim X$.

If the Chow motive $\Rs$ is indecomposable (over $\Lambda^0$), then the $\hh$-motive $\Rs_{\hh}$ is indecomposable  (over $\Lambda$).
\end{mthmB}

and also the following comparison property:

\begin{mthmC}
Suppose that $X,Y$ are generically cellular and $Y$ is a twisted form of $X$, i.e. $Y$ becomes isomorphic to $X$ over some splitting field.

If $M^{\tCH}(X,\Lambda^0)\cong M^{\tCH}(Y,\Lambda^0)$, then
$M^{\hh}(X)\cong M^{\hh}(Y)$.
\end{mthmC}

The paper is organized as follows: In section~\ref{Preliminaries} we recall concepts of an algebraic oriented cohomology theory $\hh$ of Levine-Morel and the corresponding category of $\hh$-motives.  In section~\ref{sec:filtration} we introduce the filtration on the cohomology ring $\hh(X)$ of a cellular variety $X$ which plays a central role in the paper.  In section~\ref{Group cohomology} we apply the filtration to obtain several comparison results between $\CH(G)$ and $\hh(G)$. In particular, in section~\ref{sec:comput} we compute the algebraic cobordism $\Omega$ for some groups of small ranks and construct explicit elements in the difference of between the topological and the $\gamma$-filtration on $K_0(G/B)$. In section~\ref{Motives} we apply the filtration to obtain comparison results between $\hh$-motives and Chow-motives of generically split varieties.

\medskip

\paragraph{\bf Acknowledgments}
I am grateful to my PhD supervisor Kirill Zainoulline for useful discussions concerning the subject of this paper. The work has been supported by the Ontario Trillium Graduate Scholarship, RFBR grant 12-01-33057 and by NSERC~396100-2010 and ERA grants of K.~Zainoulline.

\section{Preliminaries}\label{Preliminaries}

In the present section we recall notions of an algebraic oriented cohomology theory, a formal group law and of a cellular variety. We recall the definition of the category of $\hh$-motives with the inverted Tate object.

\subsubsection*{Oriented cohomology theories}\label{Och} 

The notion of an algebraic oriented cohomology theory was introduced by Levine-Morel~\cite{LM}  and Panin-Smirnov~\cite{PaninPush2}. Let $\Sm_k$ denote the category of smooth varieties over $\Spec k=pt$. An algebraic oriented cohomology theory $\hh^*$ is a functor from $\Sm_k^{op}$ to the category of graded rings. We will denote by $f^{\hh}\colon\hh^*(Y)\to\hh^*(X)$ the induced morphism $f\colon X\to Y$  and call it the pullback of $f$. By definition, the functor $\hh^*$ is equipped with the pushforward map $f_{\hh}\colon\hh^*(X)\to\hh^{\dim Y-\dim X+*}(Y)$ for any projective morphism $f\colon X\to Y$. These two structures satisfy the axioms of~\cite[Def.~1.1.2]{LM}. We denote its coefficient ring $\hh^*(pt)$ by $\Lambda^*$. As for the Chow groups, we will also use the lower grading for $\hh$, i.e. $\hh_i(X)=\hh^{\dim X-i}(X)$ for an irreducible variety $X$.

\subsubsection*{Formal group law}

For an oriented cohomology theory $\hh^*$ there is a notion of the first Chern class of a line bundle. For $X\in\Sm_k$ and a line bundle $L$ over $X$ it is defined as $c_1^{\hh}(L)=z^{\hh}z_{\hh}(1)\in\hh^1(X)$ where $z\colon X\to L$ is a zero section. There is a commutative associative 1-dimensional formal group law $F$ over $\Lambda^*$ such that for any two line bundles $L_1$ , $L_2$ over $X$ we have $c_1^{\hh}(L_1\otimes L_2)=F(c_1^{\hh}(L_1),c_1^{\hh}(L_2))$~\cite[Lem.~1.1.3]{LM}. We will use the notation $x+_Fy$ for $F(x,y)$. For any $x$ we will denote by $-_Fx$ the unique element such that $x+_F(-_Fx)=0$. For any $n\in\mathbb{Z}$ we will denote by $n\cdot_Fx$ the expression $x+_F\ldots +_Fx$ ($n$ times) if $n$ is positive, and $(-_Fx)+_F\ldots+_F(-_Fx)$ ($-n$ times) if $n$ is negative. By~\cite{LM} there is a natural transformation that commutes with pushforwards:
\[
\nu_X\colon\Omega^*(X)\otimes_{\mathbb{L}^*}\Lambda^*\to\hh^*(X),
\]
where $\mathbb{L}=\Omega(pt)$ is the Lazard ring and the map $\mathbb{L}^* \to \Lambda^*$ is obtained by specializing the coefficients of the universal formal group law to the coefficients of the corresponding $F$.

\subsubsection*{Cellular and generically cellular varieties}

A variety $Y\in \Sm_k$ is called cellular if there is a filtration of $Y=Y_0\supseteq Y_1\supseteq\ldots\supseteq Y_m\supseteq\emptyset$ such that each $Y_i\setminus Y_{i+1}$ is a disjoint union of affine spaces of the same rank $c_i$: $Y_i\setminus Y_{i+1}\cong\mathbb{A}^{c_i}_{k}\coprod\ldots\coprod\mathbb{A}_k^{c_i}$.

We call a variety $X$ generically cellular if $X_{k(X)}$ is a cellular variety over the function field $k(X)$.

\begin{ex}\label{flag}
Let $G$ be a split semisimple algebraic group, $B$ its Borel subgroup containing a fixed maximal split torus $T$ and  $W$ the corresponding Weyl group. For any $w\in W$ let $l(w)$ denote its length. Let $w_0\in W$ denote the longest element of $W$ and $N=l(w)$. It is well known that the flag variety $X=G/B$ has the cellular structure given by the Schubert cells $X_w$:
\[
X=X_{w_0}\supseteq \!\!\!\!\bigcup_{l(w)=N-1}\!\!\!\!X_w\supseteq \!\!\!\!\bigcup_{l(w)=N-2}\!\!\!\! X_w\supseteq\ldots\supseteq X_e=pt,
\]
where $X_w$ is the closure of $BwB/B$ in $X$. 
\end{ex}

\begin{ex}
Let $\zeta \in Z^1(k,G)$ be a 1-cocycle with values in $G$. Then the twisted form $_{\zeta}(G/B)$ of $X=G/B$ provides an example of a generically split variety.
\end{ex}

\subsubsection*{$\hh$-motives}

The notion of $\hh$-motives for the algebraic oriented cohomology theory $\hh$ was studied by Nenashev-Zainoulline in~\cite{NZ}, and Vishik-Yagita in~\cite{VY}. We refer to \cite[\S2]{VY} for definition of the category of effective $\hh$-motives. In the present paper we will deal with the category of $\hh$-motives $\Ms_{\hh}$ with the inverted Tate object. It is constructed as follows:

\medskip

Let $\SmProj_k$ denote the category of smooth projective varieties over $k$. Following~\cite{EKM} we consider the category $Corr_{\hh}$ defined as follows: For $X,Y\in\SmProj_k$ with irreducible $X$ and $m\in\mathbb{Z}$ we set
\[
Corr_m(X,Y)=\hh_{\dim X+m}(X\times Y).
\]
Objects of $Corr_{\hh}$ are pairs $(X,i)$ with $X\in \SmProj_k$ and $i\in\mathbb{Z}$. For $X\in \SmProj_k$ with irreducible components $X_l$ define the morphisms
\[
Hom_{Corr}\big((X,i),(Y,j)\big)=\coprod_{l}Corr_{i-j}(X_l,Y).
\]
For $\alpha\in Hom\big((X,i),(Y,j)\big)$ and $\beta\in Hom\big((Y,j),(Z,k)\big)$ the composition is given by the usual correspondence product: $\alpha\circ\beta=(p_{XZ})_{\hh}\big((p_{YZ})^{\hh}(\beta)\cdot (p_{XY})^{\hh}(\alpha)\big)$, where $p_{XY},p_{YZ},p_{XZ}$ denote the projections from $X\times Y\times Z$ onto the corresponding summands.

\medskip

Taking consecutive additive and idempotent completion of $Corr_{\hh}$ we obtain the category $\Ms_{\hh}$ of $\hh$-motives with inverted Tate object. Objects of this category are $(\coprod_i (X_i,n_i),p)$ where $p$ is a matrix with entries $p_{i,j}\in Corr_{n_i-n_j}(X_i,X_j)$ such that $p\circ p=p$. Morphisms between objects are given by the set
\[
Hom\big((\coprod (X_i,n_i),p),(\coprod (Y_j,m_j),q)\big)=q\circ \bigoplus_{i,j}Corr_{n_i-m_j}(X_i,Y_j)\circ p
\] 
considered as a subset of $\bigoplus_{i,j}Corr_{n_i-m_j}(X_i,Y_j)$. This is an additive category where each idempotent splits. There is a natural tensor structure inherited from the category $Corr_{\hh}$:
\[
(\coprod (X_i,n_i),p)\otimes (\coprod (Y_j,m_j),q)=(\coprod_{i,j}(X_i\times Y_j,n_i+n_j),p\times q)
\]
where $p\times q$ denotes the projector $p_{(i_1,j_1)(i_2,j_2)}=p_{i_1,i_2}\times q_{j_1,j_2}\colon X_{i_1}\times Y_{j_1}\to X_{i_2}\times Y_{j_2}$.

\medskip

There is a functor $M^{\hh}\colon \SmProj_k\to\Ms_{\hh}$ that maps a variety $X$ to the motive $M^{\hh}(X)=((X,0),id_X)$ and any morphism $f\colon X\to Y$ to the correspondence $(\Gamma_f)_{\hh}(1)\in\hh_{\dim X}(X\times Y)=Corr_0(X,Y)$, where $\Gamma_f\colon X\to X\times Y$ is the graph inclusion. We will denote by $\Delta\colon X\to X\times X$ the diagonal embedding. Then $\Delta_{\hh}(1)$ is the identity in $Corr_0(X,X)$.

Denote $M^{\hh}(pt)$  by $\Lambda$ and $((pt,1),id_{pt})$ by $\Lambda(1)$. We call $\Lambda(1)$ the $\hh$-Tate motive. We write $\Lambda(n)$ for $\Lambda(1)^{\otimes n}$ and $M^{\hh}(X)(n)$ for $M^{\hh}(X)\otimes \Lambda(n)$. The  motive $M^{\hh}(X)(n)$ is called the $n$-th twist of the motive $M^{\hh}(X)$.

By definition we have
\[
\hh^{i}(X)=Hom_{\Ms_{\hh}}(M^{\hh}(X),\Lambda(i))\text{ and }\hh_i(X)=Hom_{\Ms_{\hh}}(\Lambda(i),M^{\hh}(X)).
\]

\begin{lem}\label{Tate_basis} 
For $X\in\SmProj_k$ with the structure morphism $\pi\colon X\to pt$ any isomorphism $M^{\hh}(X)\cong\bigoplus_i\Lambda(\alpha_i)$ corresponds to a choice of two $\Lambda$-basis sets
\[
\{\tau_i\in\hh^{\alpha_i}(X)\}_i\text{ and } \{\zeta_i\in\hh_{\alpha_i}(X)\}_i
\] 
such that 
$\pi_{\hh}(\tau_i\zeta_j)=\delta_{i,j}$ in $\Lambda$ and $\sum_{i}\zeta_i\otimes\tau_i=\Delta_{\hh}(1)$ in $\hh(X\times X)$.
\end{lem}

\begin{proof}
In the direct sum decomposition $M^{\hh}(X)\cong\bigoplus_i\Lambda(\alpha_i)$ the $i$-th projection $p_i\colon M^{\hh}(X)\to\Lambda(\alpha_i)$ is defined by an element $\tau_i\in\hh^{\alpha_j}(X)$ and the $i$-th inclusion $\imath_i\colon\Lambda(\alpha_i)\to\hh(X)$ is defined by an element $\zeta_i\in\hh_{\alpha_i}(X)$. Then by definition of the direct sum we obtain 
\[
\pi_{\hh}(\tau_i\zeta_j)=\delta_{i,j}\text{ and }\sum_{i}\zeta_i\otimes \tau_i=\Delta_{\hh}(1).
\]
Let us check that $\{\zeta_i\}_i$ form a basis of $\hh({X})$. Indeed, we have $\hh(X)=\bigoplus_j\hh^j(X)=\bigoplus_j Hom_{\Ms_{\hh}}(M^{\hh}(X),\Lambda(j))\cong\bigoplus_j Hom_{\Ms_{\hh}}(\bigoplus_i\Lambda(\alpha_i),\Lambda(j))=\bigoplus_i \Lambda_{\alpha_i-*}$ and $\zeta_i$ are the images of standard generators. So $\{\zeta_i\}_i$ form a $\Lambda$-basis of $\hh(X)$. Finally, since $\{\tau_i\}_i$ are dual to $\{\zeta_i\}_i$, $\{\tau_i\}_i$ is also a basis.
\end{proof}

\begin{rem}\label{Cohomology_of_cellular_variety}
Observe that any isomorphism $M^{\hh}(X)\cong\bigoplus\Lambda(\alpha_i)$ gives rise (canonically) to an isomorphism $\hh^*(X)\cong\bigoplus_{i}\Lambda^{*-\alpha_i}$.
\end{rem}

\section{Filtration on the cohomology ring}\label{sec:filtration}

In the present section we construct a filtration on the oriented cohomology $\hh(X)$ of a cellular variety $X$ which will play an important role in the sequel.

\medskip

\begin{prop}~\label{structure} 
Assume that $X$ is a cellular variety over $k$ with the cellular decomposition $X=X_0\supseteq X_1\supseteq\ldots\supseteq X_n\supseteq \emptyset$ where $X_i\setminus X_{i+1}=\coprod_{c_i}\mathbb{A}^{\alpha_i}$. Then
\begin{itemize}
\item[(1)] the $\hh$-motive of $X$ splits as $M^{\hh}(X)=\bigoplus_i\Lambda(\alpha_i)^{\oplus c_i}$;
\item[(2)] the K\"unneth formula holds, i.e. the natural map $\hh(X)\otimes_{\Lambda}\hh(X)\to\hh(X\times X)$ is an isomorphism;
\item[(3)] the specialization maps $\nu_{X}\colon\Omega(X)\otimes\Lambda\to\hh({X})$ and $\nu_{{X}\times{X}}\colon\Omega({X}\times{X})\otimes\Lambda\to\hh({X}\times{X})$ are isomorphisms.
\end{itemize}
\end{prop}

\begin{proof}
By~\cite[Cor.~66.4]{EKM} the Chow motive $M^{\tCH}({X})$ splits, then~\cite[Cor. 2.9]{VY} implies that the motive $M^{\Omega}({X})$ splits into a sum of twisted Tate motives $M^{\Omega}({X})=\bigoplus_{i\in I}\mathbb{L}(\alpha_i)^{\oplus c_i}$. By~Lemma~\ref{Tate_basis} there are elements $\zeta_{i,j}^{\Omega}\in\hh_{\alpha_i}({X})$ and $\tau_{i,j}^{\Omega}\in\hh^{\alpha_i}({X})$, $j\in\{1..c_i\}$ such that $\pi_{\Omega}(\zeta_{i,j}^{\Omega}\tau_{i',j'}^{\Omega})=\delta_{(i,j),(i',j')}$ and $\Delta_{\Omega}(1)=\sum_{i,j}\zeta_i^{\Omega}\otimes\tau_i^{\Omega}$. Denote $\zeta_{i,j}^{\hh}=\nu(\zeta_{i,j}^{\Omega}\otimes 1)$ and $\tau_{i,j}^{\hh}=\nu(\tau_{i,j}^{\Omega}\otimes 1)$. Since $\nu$ commutes with pullbacks and pushforwards, $\pi_{\hh}(\zeta_{i,j}^{\hh}\tau_{i',j'}^{\hh})=\delta_{(i,j),(i',j')}$ and $\Delta_{\hh}(1)=\sum_{i,j}\zeta_{i,j}^{\hh}\otimes\tau_{i,j}^{\hh}.$ Then by~Lemma~\ref{Tate_basis} we have $M^{\hh}({X})=\bigoplus_i\Lambda(\alpha_i)^{\oplus c_i}$, so $(1)$ holds. 

\medskip

The K\"unneth map fits into the diagram
\[
\begin{xymatrix}{
\hh({X})\otimes_{\Lambda}\hh({X})\ar[r]\ar@{=}[d] & \hh({X}\times{X})\ar@{=}[d]\\
\left(\bigoplus_i\Lambda^{*-\alpha_i}\right)\otimes_{\Lambda}(\bigoplus_j\Lambda^{*-\alpha_j})\ar[r] & \bigoplus_{i,j}\Lambda^{*-\alpha_i-\alpha_j}}
\end{xymatrix}
\]
where the bottom arrow is an isomorphism, so the K\"unneth formula $(2)$ holds.

\medskip

Note that the natural map $\nu_{{X}}$ can be factored as follows 
\[
\nu_{{X}}\colon\Omega({X})\otimes\Lambda=\oplus_{m}Hom_{\Ms_{\Omega\otimes\Lambda}}(\bigoplus\Lambda(\alpha_i),\Lambda(m))\to
Hom_{\Ms_{\hh}}(\bigoplus\Lambda(\alpha_i),\Lambda(m))=\hh({X}).
\]
Thus $\nu_{X}$ is an isomorphism. The same reasoning proves the statement for $\nu_{X\times X}$, hence, $(3)$ holds.
\end{proof}

\begin{dfn}\label{filtration}
Let $X$ be a cellular variety. Consider two basis sets $\zeta_i\in \hh_{\alpha_i}(X)$ and $\tau_i\in\hh^{\alpha_i}(X)$ provided by Proposition~\ref{structure} and Lemma~\ref{Tate_basis}. We define the filtration $\hh^{(l)}({X})$ as the $\Lambda$-linear span
\[
\hh^{(l)}({X})=\bigoplus_{N-\alpha_i\geqslant l}\Lambda\zeta_{i}=\bigoplus_{\alpha_i\geqslant l}\Lambda\tau_{i}.
\]
We denote $\hh^{(l/l+1)}({X})=\hh^{(l)}({X})/\hh^{(l+1)}({X})$ and $Gr^*\hh(X)=\bigoplus_{l}\hh^{(l/l+1)}(X)$ to be the associated graded group. Lemma~\ref{filtr_prod} implies that the latter is a graded ring.
\end{dfn}

\medskip

\begin{rem}~\label{topological_filtration}
In the case when the theory $\hh$ is generically constant and satisfies the localization property, the filtration introduced above coincides with the topological filtration on $\hh(X)$, i.e. with the filtration where the $l$-th term is generated over $\Lambda$ by classes $[Z\to X]$ of projective morphisms $Z\to X$ birational on its image and $\dim X-\dim Z\leqslant l$. This fact follows from the generalized degree formula~\cite[Thm.~4.4.7]{LM}.
\end{rem}

\medskip

\begin{lem}~\label{filtr_prod} 
$\hh^{(l_1)}({X})\cdot\hh^{(l_2)}({X})\subseteq\hh^{(l_1+l_2)}({X})$.
\end{lem}

\begin{proof}
We have  $\tau^{\Omega}_i\tau^{\Omega}_j=\sum_l a_l\zeta^{\Omega}_l$ in $\Omega({X})$ for some $a_l\in\mathbb{L}$. Then $\alpha_i+\alpha_j=\deg(a_l)+\alpha_l$. Since $\deg(a_l)\leqslant 0$,  $\alpha_l\geqslant \alpha_i+\alpha_j\geqslant l_1+l_2$ for any nontrivial $a_l$. Since $\zeta^{\hh}_i=\nu(\zeta^{\Omega}_i\otimes 1)$ we have $\zeta^{\hh}_i\zeta^{\hh}_j=\sum_l (a_l\otimes 1)\zeta^{\hh}_l$ with $\alpha_l\geqslant \alpha_i+\alpha_j\geqslant l_1+l_2$. So $\zeta^{\hh}_i\zeta^{\hh}_j\in\hh^{(l_1+l_2)}({X})$.
\end{proof}

\begin{prop}\label{Psi_existence} For a cellular $X$ there is a graded ring isomorphism:
\[
\Psi\colon\bigoplus_{i=0}^{N}\hh^{(i/i+1)}({X})\to\CH({X},\Lambda).
\]
\end{prop}

\begin{proof}
By~Proposition~\ref{structure} it is sufficient to prove the statement for $\hh=\Omega$. Observe that $\Omega^{(l/l+1)}({X})$ is a free $\mathbb{L}$-module with the basis $\tau_{i}^{\Omega}+\hh^{(l+1)}({X})$ with $\alpha_i=l$ and $\CH^i({X},\mathbb{L})$ is a free $\mathbb{L}$-module with basis $\tau^{\tCH}_i$ with $\alpha_i=l$. Thus the $\mathbb{L}$-module homomorphism $\Psi_l$ defined by
\[
\Psi_l(\tau_{i}^{\Omega}+\hh^{(i+1)}({X}))=\tau^{\tCH}_i
\]
is an isomorphism.

Let us check that $\Psi=\bigoplus\Psi_l$ preserves multiplication. For any $i,j$ we have 
\[
\tau^{\Omega}_i\tau^{\Omega}_j=\sum_m a_m\tau^{\Omega}_m \eqno{(*)}
\] 
for some $a_m\in\mathbb{L}$. Then for any $m$ we have $\deg(a_m)+\alpha_m=\alpha_i+\alpha_j$. Then in $\hh^{(\alpha_i+\alpha_j/\alpha_i+\alpha_j+1)}$ we have 
\[
\tau^{\Omega}_i\tau^{\Omega}_j=\sum_{\alpha_m=\alpha_i+\alpha_j} a_m\tau^{\Omega}_m \text{ modulo }+\hh^{(\alpha_i+\alpha_j+1)}({X})
\]
Observe that $\mathbb{L}^0=\mathbb{Z}$ and for all $a_m\in \mathbb{L}$ such that $\deg(a_m)<0$ we have that $a_m\otimes 1_{\mathbb{Z}}=0$ in $\mathbb{Z}$. Thus tensoring ($*$) with $1_{\mathbb{Z}}$ we get
\[
\tau^{\tCH}_i\tau^{\tCH}_j=\sum_{\alpha_m=0} (a_m\otimes 1)\tau^{\tCH}_m.
\]
So $\Psi_{\alpha_i+\alpha_j}(\tau^{\Omega}_i+\hh^{(\alpha_i+1)}({X})\cdot \tau_j^{\Omega}+\hh^{(\alpha_j+1)}({X}))=\tau^{\tCH}_i\cdot\tau_j^{\tCH}$. Hence, $\Psi$ is a graded ring isomorphism.
\end{proof}

\begin{lem}
$\Psi(\zeta_i^{\hh}+\hh^{(\alpha_i+1)}({X}))=\zeta_i^{\tCH}$.
\end{lem}

\begin{proof}
It is sufficient to show the statement for $\hh=\Omega^*$. Consider the expansion $\zeta_i^{\Omega}=\sum a_j\tau_j^{\Omega}$ for some $a_j\in\mathbb{L}$ with $\deg a_j+\alpha_j=N-\alpha_i$. Since $\deg a_j\leqslant 0$ we have  
\[
\zeta_i^{\Omega}=\sum_{\deg a_j=0}a_j\tau_j^{\Omega} \mod \Omega^{(N-\alpha_i+1)}({X}).
\] 
Therefore, $\Psi(\zeta_i^{\Omega}+\Omega^{(N-|w|+1)}({X}))=\zeta_i^{\tCH}$.
\end{proof}

\section{Oriented cohomology of a group}\label{Group cohomology}

In the present section, using the filtration introduced in~\ref{filtration} we
compute algebraic cobordism for some groups of small ranks and for
$PGL_n$, $n\ge 2$. We also construct nontrivial elements in the
difference between the topological and the $\gamma$-filtration on $K_0(G/B)$.

\medskip

In this section we assume that the associated to $\hh$ weak Borel-Moore
homology theory satisfies the localization property of~\cite[Definition 4.4.6]{LM}. Examples of such theories include $\hh(-)=\Omega(-)\otimes\Lambda$, or any oriented cohomology theory in the sense of Panin-Smirnov~\cite{PaninPush2}.

\medskip

Consider the variety $X=G/B$ where $G$ is a split semisimple algebraic
group. Let $\pi_{G/B}\colon G\to X$ be the quotient map. According to Example~\ref{flag} $X$ is cellular. For any $w\in W$ we fix a minimal decomposition $w=s_{i_1}\ldots s_{i_m}$ into simple reflections. Denote the corresponding multiindex by $I_w=(i_1,\ldots i_m)$ and consider the Bott-Samelson variety $X_{I_w}/B$~\cite[\S 11]{CPZ}. Then $p_{I_w}\colon X_{I_w}/B\to G/B$ is a desingularisation of the Schubert cell $X_w.$ Take 
\[
\zeta_{w}=(p_{I_w})_{\hh}(1)\in\hh_{l(w)}(X)=Hom_{\Ms_{\hh}}(\Lambda(l(w)),M^{\hh}(X)
\] 
to be the embedding in the direct sum decomposition $\bigoplus_{w\in
  W}\Lambda(l(w))\cong M^{\hh}(G/B)$. So, with this choice of
isomorphism $M^{\hh}(G/B)\cong\bigoplus_{w\in W}\Lambda(l(w))$ the
basis given by Lemma~\ref{Tate_basis} coincides with the basis $\zeta_{I_w}$ constructed in~\cite[\S 13]{CPZ}. 

\medskip

Let $\Lambda[[T^*]]_F$ be the formal group algebra introduced
by Calm\`es-Petrov-Zainoulline in \cite[\S2]{CPZ}, where
$T^*$ is the character lattice of $T$ and $F$ is the formal group law
of the theory $\hh$. There is the characteristic map $\cs_F\colon\Lambda[[T^*]]_F\to\hh(G/B)$ such that $\cs_F(x_{\lambda})=c_1^{\hh}(\mathcal{L(\lambda)})$ for a character $\lambda$. By~\cite[Prop.~5.1]{GZ} there is a short exact sequence
\begin{equation}\label{eq:group_sequence}
0\to\cs(I_F)\to\hh(G/B)\stackrel{\pi_{G/B}^{\hh}}\longrightarrow\hh(G)\to 0,
\end{equation}
where $I_F$ denotes the ideal in $\Lambda[[T^*]]_F$ generated by $x_{\alpha}$ for $\alpha\in T^*.$
By~\cite[Lem.~4.2]{CPZ} there is a graded algebras isomorphism $\psi\colon \bigoplus_{m=0}^{\infty} I_F^{m}/I_F^{m+1}\to S^*_{\Lambda}(T^*)$ where $S^*_{\Lambda}(T^*)$ denotes the symmetric algebra over $T^*$. Let $F_a$ denote the additive formal group law. We will need the following
\begin{lem}
The following diagram commutes:
\[
\begin{xymatrix}{
\bigoplus_{m=0}^{\infty} I_F^{m}/I_F^{m+1}\ar[r]^-{Gr\cs_F}\ar[d]^{\psi} & \bigoplus_{m=0}^{\infty} \hh^{(m/m+1)}(G/B)\ar[d]^{\Psi}\\
S^*_{\Lambda}(T^*)\ar[r]^{\cs_{F_a}} & \CH(G/B,\Lambda)
}
\end{xymatrix}
\]
\end{lem}
\begin{proof}
By definition, it is sufficient to prove that $\Psi(pr_1(c^{\hh}_1(\mathcal{L}_{\lambda})))=c^{\tCH}_1(\mathcal{L}_{\lambda})$
Consider the expansion $c_1^{\Omega}(\mathcal{L}_{\lambda})=\sum
r_w\tau^{\Omega}_w$ we have that $r_w\in\mathbb{L}^0=\mathbb{Z}$ for
$|w|=1$. Then $c_1^{\hh}(\mathcal{L}_{\lambda})=\sum (r_w\otimes
1_{\Lambda})\tau_w^{\hh}$ and \[
\Psi(pr_1(c_1^{\hh}(\mathcal{L}_{\lambda})))=\sum_{|w|=1}
r_w\tau^{\tCH}_w=c_1^{\Omega}(\mathcal{L}_{\lambda})\otimes
1_{\mathbb{Z}}=c_1^{\tCH}(\mathcal{L}_{\lambda}). \qedhere\]
\end{proof}

\begin{lem}\label{addker}
For the additive group law the induced filtration satisfies
\[
\cs(I_{F_a})\CH(X,\Lambda)\cap \CH^i(X,\Lambda)=\sum\cs(x_{\alpha})\CH^{i-1}(X,\Lambda).
\]
\end{lem}
\begin{proof}
Note that $\cs(I_a)\CH(X,\Lambda)$ is generated by the elements
$\cs(x_{\alpha})\in\CH^1(X,R)$. For any element
$z=\sum\cs(x_{\alpha})y_{\alpha}$ of $\cs(I_a)\CH(X,\Lambda)$ we have
that $z$ lies in the $\CH^{i}(X,\Lambda)$ if and only if $y_{\alpha}\in \CH^{i-1}(X,\Lambda)$. Therefore $\cs(I_a)\CH(X,\Lambda)\cap \CH^i(X,\Lambda)=\sum\cs(x_{\alpha})\CH^{i-1}(X,\Lambda)$.
\end{proof}

Let $\hh^{(i)}(G)$ denote the image 
$\pi_{G/B}^{\hh}(\hh^{(i)}(G/B))$ and let
$\hh^{(i/i+1)}(G)$ denote the quotient $\hh^{(i)}(G)/\hh^{(i+1)}(G)$.

\begin{prop}\label{GroupComparison}
For every $i$ there is an exact sequence:
\[
0\rightarrow\tfrac{\cs(I)\hh^{(i-1)}(X)}{\hh^{(i+1)}(X)}\rightarrow \tfrac{(\cs(I)\hh(X))\cap\hh^{(i)}(X)}{\hh^{(i+1)}(X)}\rightarrow \CH^i(G,\Lambda)\rightarrow \hh^{(i/i+1)}(G)\rightarrow 0.
\]
\end{prop}
\begin{proof}
By~\cite[Prop.~2, \S2.4]{Bourbaki} we obtain from \eqref{eq:group_sequence} the short exact sequence:
\[
0\rightarrow \tfrac{(\cs(I)\hh(X))\cap\hh^{(i)}(X)}{\hh^{(i+1)}(X)}\rightarrow \tfrac{\hh^{(i)}(X)}{\hh^{(i+1)}(X)}\rightarrow \hh^{(i/i+1)}(G)\rightarrow 0.
\]

By Lemma~\ref{addker} applied to the case of additive formal group
law, the above sequence turns into
\[
0\rightarrow \sum\cs(x_{\alpha})\CH^{i-1}(X,R)\rightarrow \CH^i(X,R)\rightarrow \CH^i(G,R)\rightarrow 0.
\]

Observe that for isomorphism $(\Psi^i)^{-1}$
we have \[(\Psi^i)^{-1}\left(\sum\cs(x_{\alpha})\CH^{i-1}(X,\Lambda)\right)=\left(\tfrac{\cs(I)\hh^{(i-1)}(X)}{\hh^{(i+1)(X)}}\right)\subseteq \left(\tfrac{\cs(I)\hh^(X)\cap\hh^{(i)}(X)}{\hh^{(i+1)(X)}}\right)\eqno{(*)}\]

Then we get the following diagram with exact rows:
\[
\xymatrix{
0\ar[r] & \sum\cs(x_{\alpha})\CH^{i-1}(X,\Lambda)\ar[r]\ar@{^{(}_->}[d] & \CH^i(X,\Lambda)\ar[r]\ar@{=}[d]^{(\Psi^i)^{-1}} & \CH^i(G,\Lambda)\ar[r] & 0\\
0\ar[r] & \frac{\cs(I)\hh(X)\cap\hh^{(i)}(X)}{\hh^{(i+1)(X)}}\ar[r] & \hh^{(i/i+1)}(X)\ar[r] & \hh^{(i/i+1)}(G)\ar[r] & 0
}
\]
The latter sequence and ($*$) gives rise to the exact sequence
\[
0\rightarrow\tfrac{\cs(I)\hh^{(i-1)}(X)}{\hh^{(i+1)}(X)}\rightarrow \tfrac{(\cs(I)\hh(X))\cap\hh^{(i)}(X)}{\hh^{(i+1)}(X)}\rightarrow \CH^i(G,\Lambda)\rightarrow \hh^{(i/i+1)}(G)\rightarrow 0.\qedhere
\]
\end{proof}
\begin{cor}
Assume that pullbacks $p_{G/B}^{\tCH}(X_{w_i})$ of Schubert cells generate $\CH(G)$ as $\mathbb{Z}$-algebra for some elements $w_1,\ldots w_m$ in $W$. Then the pullbacks of Schubert cells $p_{G/B}^{\hh}(\zeta_{w_i})$ generate $\hh(G)$ as $\Lambda$-algebra.
\end{cor}
\begin{proof}
By~\ref{GroupComparison} classes of $p_{G/B}^{\hh}(\zeta_{w_i})$ generate the associated graded ring $\bigoplus_i\hh^{(i/i+1)}(G)$. Then $p_{G/B}^{\hh}(\zeta_{w_i})$ generate $\hh(G)$, since the filtration is finite.
\end{proof}

The following observations will be useful for computations
\begin{lem}\label{CH0}
If $\CH^i(G)=0$ then $\frac{\cs(I)\hh^{(i-1)}(G/B)}{\hh^{(i+1)}(G/B)}=\frac{\hh^{(i)}(G/B)}{\hh^{(i+1)}(G/B)}$
\end{lem}

\begin{lem}
Assume that $\CH^1(G)=0$. Then for $i\geqslant 2$ 
\[\cs(I)\hh(G/B)\cap\hh^{(i)}(G/B)=\cs(I)\hh^{(1)}(G/B)\cap\hh^{(i)}(G/B)\]
\end{lem}

\begin{proof}
Note that ideal $\cs(I)\hh(G/B)$ is generated by $\cs(x_{\alpha})$ where $\alpha$ runs over the basis of the character lattice.
\end{proof}

\section{Examples of computations}\label{sec:comput}
The results of the previous section allow us to obtain some
information concerning the ring $\hh(G)$ from $\CH(G)$. Moreover, in
some cases it allows us to compute $\hh(G)$.

\medskip

We follow the notation of the previous section.
We denote $\bigoplus\hh^{(i/i+1)}(G)$ by $Gr^*\hh(G)$. For
$a\in\hh^{i}(G)$ let $\overline{a}\in\hh^{(i/i+1)}(G)$ denote its
residue class. Let $\alpha$ be
the projection $\CH^*(G,\Lambda)\to\oplus Gr^*\hh(G)$.

\subsubsection*{Algebraic cobordism of $G_2$.}
According to~\cite{Marlin}
we have \[\CH^*(G_2,\mathbb{Z})=\frc{\mathbb{Z}[x_3]}{(x_3^2,2x_3)}\]
where $x_3=\pi^{\tCH}(\zeta_{212})$, $\pi\colon G\to G/B$ is the projection
and $\zeta_{212}$ is the Schubert cell corresponding to the word $w=s_2s_1s_2$.
Let $y_3$ denote the pullback $\pi^{\Omega}(\zeta^{\Omega}_{212})$ of the corresponding Schubert cell in the ring $\Omega(G/B)$ (see Theorem 13.12 of~\cite{CPZ}). Observe that $\alpha^3(x_3)=\overline{y_3}$.
Since $\CH(G_2,\mathbb{L})$ is generated by $1$ and $x_3$, $Gr^*\Omega(G_2)$ is generated by $1$ and $\overline{y_3}$.
Then by\cite[\S 2.8]{Bourbaki} $\Omega(G_2)$ is generated by $1$ and $y_3\in\Omega^{(3)}(G_2)$. 
Since $2x_3=0,$ then $\overline{2y_3}$ so $2y_3\in\Omega^{(4)}(G_2)$
which is zero since $\CH^i(G_2)=0$ for $i\geqslant 4$. Thus $2y_3=0$
and $y_3^2\in\Omega^{(6)}(G_2)=0$.

Let us now compute $\Omega^{(3)}(G_2)$. Proposition~\ref{GroupComparison} gives us the exact sequence
\[
0\rightarrow \tfrac{\cs(I)\Omega^{(2)}(G_2/B)}{\Omega^{(4)}(G_2/B)}\rightarrow \tfrac{\cs(I)\Omega(G_2/B)\cap\Omega^{(3)}(G_2/B)}{\Omega^{(4)}(G_2/B)}\rightarrow \mathbb{L}/2\cdot x_3\rightarrow \Omega^{(3)}(G_2).
\]
Note that since $\cs(I)\Omega(G/B)$ is generated by $\cs(x_1)$ and $\cs(x_2)$. By Lemma~\ref{CH0} we have
\[
\tfrac{\cs(I)\Omega(G_2/B)}{\Omega^{(4)}(G_2/B)}=\tfrac{\langle\zeta_{12121},\zeta_{21212}\rangle\Omega(G_2/B)}{\Omega^{(4)}(G_2/B)}
\]

Then $\cs(I)\cap\Omega^{(3)}(G_2/B)/\Omega^{(4)}(G_2/B)$ equals to the set
\[
\{x=a\zeta_{12121}+b\zeta_{21212}\mid x\in\Omega^{(3)}\}.
\]
It is enough to consider only $a,b$ in $\Omega^{(1)}(G_2/B)\setminus\Omega^{(2)}(G_2/B)$ since for $a,b\in\Omega^{(0)}(G_2/B)\setminus\Omega^{(1)}(G_2/B)$ we have $x\notin\Omega^{(2)}(G_2/B)$.
So we consider
\[
a=r_1\zeta_{12121}+r_2\zeta_{21212} \text{ and } b=s_1\zeta_{12121}+s_2\zeta_{21212} \text{ for } r_1,r_2,s_1,s_2\in\mathbb{L}.
\]
Using the multiplication table for $G_2/B$ from~\cite{CPZ} we obtain that $x$ equals
\[
(r_2+s_1+s_2)\zeta_{1212}+(3r_1+r_2+s_1)\zeta_{2121}+(r_2+s_1+3r_1)a_1\zeta_{121}+(r_2+s_1)a_1\zeta_{212}
\]
modulo $\Omega^{(4)}(G_2/B)$.
Then $x\in\Omega^{(3)}(G_2/B)/\Omega^{(4)}(G_2/B)$ iff $r_2+s_1+s_2=0$ and $3r_1+r_2+s_1=0.$
Therefore, $r_2+s_1=-3r_1$ and $s_2=3r_1$. 
So \[x+\Omega^{(4)}(G_2/B)=-3r_1a_1\zeta_{212}+\Omega^{(4)}(G_2/B).\]
Hence, the kernel of $\mathbb{L}\cdot 2x_3\to\Omega^{(3)}(G_2)$ is generated by $3a_1x_3$.
Then \[\Omega^{(3)}(G_2)=\mathbb{L}/(2,3a_1)\cdot y_3=\mathbb{L}/(2,a_1)\cdot y_3\]
and we obtain that 
\begin{equation}
\Omega(G_2)=\frc{\mathbb{L}[y_3]}{(y_3^2,2y_3,a_1y_3)}.
\end{equation}

Observe that
taking the latter equality modulo 2 we obtain the result established
by Yagita in~\cite{Yagita}.

\subsubsection*{Algebraic cobordism of groups $SO_{n}$, $Spin_{m}$ for
  $n=3,4$ and $m=3,4,5,6$}

According to~\cite{Marlin} 
\[\CH(Spin_i)=\mathbb{Z} \text{ for }
i=3,4,5,6.
\] Then by Proposition~\ref{GroupComparison} we obtain 
\begin{equation}
\Omega(Spin_i)=\mathbb{L}\text{ for }i=3,4,5,6.
\end{equation}

We have $\CH(SO_3)=\mathbb{Z}[x_1]/(2x_1,x_1^2)$ where $x_1=\pi^{\tCH}(\zeta_{w_0s_1})$. Since $CH^i(SO_3)=0$ for $i\geqslant 2$, $\Omega^{(2)}(SO_3)=0$.
For $i=1$ two left terms of exact sequence
of~Proposition~\ref{GroupComparison} coincide, so there is an
isomorphism $\CH^1(SO_3)\to\Omega^{(1)}(SO_3)$. Hence, we obtain 
\begin{equation}
\Omega(SO_3)=\frc{\mathbb{L}[y_1]}{(2y_1,y_1^2)},\;\text{ where }y_1=\pi^{\Omega}(\zeta_{w_0s_1}).
\end{equation}
Since $\CH(SO_4)=\mathbb{Z}[x_1]/(2x_1,x_1^2)$ the same reasoning
proves that \begin{equation}
\Omega(SO_4)=\frc{\mathbb{L}[y_1]}{(2y_1,y_1^2)}.
\end{equation}

\subsubsection*{Oriented cohomology of $PGL_n$}
\begin{lem}
For any oriented cohomology theory $\hh$ with the coefficient ring
$\Lambda$ and the formal group law $F$ we have
\[\hh(PGL_{n})=\frc{\Lambda[x]}{(x^n,nx^{n-1},\ldots{n\choose d}x^{d},\ldots, nx, \ n\cdot_Fx)}.\]
\end{lem}
\begin{proof}
Consider the variety of complete flags $X=SL_n/B$. Let $F_i$ denote the tautological vector bundle of dimension $i$ over $X$.
Let $L_1=F_1$ and $L_i=F_i/F_{i-1}$ for $i=2,\ldots n$. Then, by~\cite[Thm.~2.6]{HK} we have
\[\hh(X)\cong\Lambda[x_1,\ldots x_n]/S(x_1,\ldots x_n)\eqno(*)\]
where $S(x_1,\ldots x_n)$ denotes the ideal generated by positive degree symmetric polynomials in variables $x_1,\ldots x_n$,
and the isomorphism sends $x_i$ to the Chern class $c^{\hh}_1(L_i)$. 
The maximal split torus $T\subseteq SL_n$ consists of diagonal
matrices with trivial determinant. Let $\chi_i\in\hat{T}$ denote the
character that sends the diagonal matrix to its $i$-th entry. So, the
character lattice equals to
$M=\mathbb{Z}\chi_1\oplus\ldots\oplus\mathbb{Z}\chi_n/(\chi_1+\ldots +
\chi_n)$. Observe that $L_i$ coincides with the line bundle
$\mathcal{L}(\chi_i)$, so by definition we have that $x_i=\cs(\chi_i)$, where $\cs\colon\Lambda[[M]]_F\to\hh(X)$ is the characteristic map.
Note that the roots of $PGL_n=PSL_n/\mu_n$ are equal
$n\chi_1,\chi_2-\chi_1,\ldots, \chi_n-\chi_1$. 

According to~\cite[5.1]{GZ} we have
\[\hh(PGL_n)=\hh(X)/(\cs(n\chi_1),\cs(\chi_2-\chi_1),\ldots,\cs(\chi_n-\chi_1)).\] 
Then in the quotient we have
\[
\overline{\cs(\chi_i)}=\overline{\cs(\chi_1+\chi_i-\chi_1)}=\overline{\cs(\chi_1)+_F\cs(\chi_i-\chi_1)}=\overline{\cs(\chi_1)}.\] 
Taking
$x=\overline{\cs(\chi_1)}$ by $(*)$ we get
\[\hh(PGL_n)=\Lambda[x]/(S(x,\ldots,x), n\cdot_F x).\]
According to~\cite{HK} $S(x_1,\ldots x_n)$ is generated by polynomials
$f_n(x_n)$,
$f_{n-1}(x_n,x_{n-1})$, $\ldots$, $f_1(x_n,\ldots, x_i)$ where $f_i(x_n,\ldots,x_i)$ denotes the sum of all degree $i$ monomials in $x_n,\ldots, x_i$. Note that ${n\choose d}$ equals to the number of degree $d$ monomials in $n-d+1$ variables. Then substituting $x_1=\ldots x_n=x$ we obtain that $x^n,nx^{n-1},{n\choose d}x^{n-1},\ldots, nx$ generate the ideal $S(x,\ldots,x)$.
\end{proof}

\begin{ex}
For a prime number $p$ and $0<d<p$ the coefficient 
${p\choose d}$ is divisible by $p$.
By~\cite[Rem.~5.4.8]{Hazewinkel} over $\Lambda/p\Lambda$ we have
$p\cdot_Fx=p\beta_0(x)+\beta_1(x^p)$. Thus, the ideal
$I=(x^p,px^{p-1},\ldots{p\choose d}x^{d},\ldots, px, \ p\cdot_Fx)$ is
generated by $x^p$, $px$. So for any prime $p$ we
have \[\hh(PGL_p)=\frc{\Lambda[x]}{(px,x^p)}.\] In the case $\hh=K_0$
this agrees with \cite[3.6]{ZGamma}.
\end{ex}

\subsubsection*{Topological and the $\gamma$-filtration}\label{gammafiltr}
Proposition~\ref{GroupComparison} allows to estimate the difference
between the topological and the Grothendieck $\gamma$-filtration on
$K_0(G/B)$ for a split linear algebraic group $G$.
Namely, consider two filtrations on $K_0(G/B)$:
\begin{itemize}
\item[] $\gamma$-filtration: $\gamma^i(G/B)=\langle c_1(\mathcal{L}(\lambda))\mid \lambda\in T^*\rangle$~\cite[Definition 4.2]{ZGamma},
\item[] topological filtration: $\tau^i(G/B)=\langle[\mathcal{O}_V]\mid\codim(V)\geqslant i\rangle$.
\end{itemize}

\begin{prop}
Let $G$ be a split semisimple simply connected linear algebraic group
such that $\CH^i(G)=0$ for $1\leqslant i\leqslant n-1$ and
$\CH^n(G)\neq 0$.
Let $\zeta_w$ be a Schubert cell such that $\pi^{\tCH}(\zeta)$ is nontrivial in $\CH^n(G)$. 

Then $\gamma^i(G/B)+\tau^{i+1}(G/B)=\tau^i(G/B)$ for $i<n$ and the class of $\zeta_w^{K_0}$ is nontrivial in $\tau^n(G/B)/\gamma^n(G/B)$.
\end{prop}
\begin{proof}
As shown in~\cite{Panin} $K_0(G)=\mathbb{Z}$ for a simply connected group $G$. Then
characteristic map $\cs$ is surjective~\cite[\S 1B]{GarZ}. We have $K_0^{(1)}(G/B)=\tau^1=\gamma^1$.
Note that $K_0^{(i)}(G/B)=\tau^i.$ Then $\gamma^1\tau^0\cap\tau^i=\tau^i$
and the Proposition~\ref{GroupComparison} gives us a short exact sequence for all $i\geqslant 1$:
\[
0\to\tfrac{\gamma^1\tau^{i-1}}{\tau^{i+1}}\to \tfrac{\tau^i}{\tau^{i+1}}\to\CH^i(G,\mathbb{Z}[\beta,\beta^{-1}])\to 0.
\]
Then for any $1\leqslant i<n$ we have $\frc{\gamma^1\tau^{i-1}}{\tau^{i+1}}=\frc{\tau^i}{\tau^{i+1}}.$
By induction we get $\tau^i=\gamma^i+\tau^{i+1}$ for $i<n$ and for $i=n$ we get
By induction we get $\tau^i=\gamma^i+\tau^{i+1}$ for $i<n$ and for $i=n$ we get
\[
0\to\tfrac{\gamma^n}{\tau^{n+1}}\to \tfrac{\tau^n}{\tau^{n+1}}\to\CH^n(G,\mathbb{Z}[\beta,\beta^{-1}])\to 0.
\]
So for any nontrivial element of of $\CH^n(G)$ the class of its preimage is nontrivial in $\tau^{n}/\gamma^n$.
\end{proof}

\section{Applications to $\hh$-motivic decompositions}\label{Motives}
Throughout this section we consider a generically cellular variety $X$
of dimension $N$ and an oriented cohomology theory $\hh^*$ that is
generically constant and is associated with weak Borel-Moore homology
$\hh_*$ which satisfies the localization property. 
These assumptions imply that the generalized degree formula of Levine-Morel~\cite[Theorem 4.4.7]{LM} holds. 
The aim of this section is to prove theorems A, B  and C of
the introduction which provide a comparison between the Chow motive
$M(X)$ and the $\hh$-motive $M^{\hh}(X)$ of $X$. 

\medskip

Let $L$ be the splitting field of $X$ and $\overline{X}=X\times_{\Spec k}\Spec L$. Let $p$ denote the projection $p\colon\overline{X}\times\overline{X}\to X\times X$. Since $\overline{X}$ is cellular, we may consider a filtration on $\hh(\overline{X})$ introduced in~\ref{filtration}. It gives rise to a filtration on $\hh(\overline{X}\times\overline{X})=\hh(\overline{X})\otimes_{\Lambda}\hh(\overline{X})$.
Namely, we set 
\[
\hh^{(l)}(\overline{X}\times\overline{X})=\sum_{i+j=l}\hh^{(i)}(\overline{X})\otimes_{\Lambda}\hh^{(j)}(\overline{X}).
\]
On $\hh(X\times X)$ we consider the induced filtration 
\[\hh^{(l)}(X\times X)=(p^{\hh})^{-1}(\hh^{(l)}(\overline{X}\times\overline{X})).\]  

Denote the quotient
$\hh^{(l)}(\overline{X}\times\overline{X})/\hh^{(l+1)}(\overline{X}\times\overline{X})$
by $\hh^{(l/l+1)}(\overline{X}\times\overline{X})$ and denote by
$pr_{l}\colon\hh^{(l)}(\overline{X}\times\overline{X})\to\hh^{(l/l+1)}(\overline{X}\times\overline{X})$
the usual projection.
Denote \[
\hh^{(i)}_{2N-i}(X\times X)=\hh^{(i)}(X\times X)\cap\hh_{2N-i}(X\times
X)\text{ and}\] \[\hh^{(i)}_{2N-i}(\overline{X}\times
\overline{X})=\hh^{(i)}(\overline{X}\times
\overline{X})\cap\hh_{2N-i}(\overline{X}\times \overline{X}).
\]

\begin{lem}
There is a graded ring isomorphism 
\[
\Phi\colon\bigoplus_{i=0}^{2N}\hh^{(i/i+1)}(\overline{X}\times\overline{X})\to\CH^*(\overline{X}\times\overline{X},\Lambda).
\]
\end{lem}

\begin{proof}
Since
\[\bigoplus_{i=0}^{2N}\hh^{(i/i+1)}(\overline{X}\times\overline{X})=\bigoplus_{i=0}^{N}\hh^{(i/i+1)}(\overline{X})\otimes_{\Lambda}\bigoplus_{i=0}^{N}\hh^{(i/i+1)}(\overline{X})\]
and
$\CH(\overline{X}\times\overline{X},\Lambda)=\CH(\overline{X},\Lambda)\otimes_{\Lambda}\CH(\overline{X},\Lambda)$
take $\Phi=\Psi\otimes\Psi$, where $\Psi$ is defined in \ref{Psi_existence}.
\end{proof}

\begin{rem}
The restriction of $\Phi_i$ gives an isomorphism $\Phi_i\colon\hh^{(i/i+1)}_{2N-i}(\overline{X}\times\overline{X})\to\CH^i(\overline{X}\times\overline{X},\Lambda^{0})$.
\end{rem}

The following lemma provides an $\hh$-version of the Rost Nilpotence Theorem:

\begin{lem}\label{Rost}
The kernel of the pullback map $p^{\hh}\colon End(M^{\hh}(X))\to End(M^{\hh}(\overline{X}))$ consists of nilpotents.
\end{lem}

\begin{proof}
Consider a diagram
\[
\begin{xymatrix}{
End(M^{\Omega}(X))\ar[d]\ar[r]^{p^{\Omega}} & End(M^{\Omega}(\overline{X}))\ar[d] \\
End(M^{\tCH}(X))\ar[r] & End(M^{\tCH}(\overline{X}))
}\end{xymatrix}
\]
where vertical arrows are ring homomorphisms that arise from the
canonical map $\Omega(-)\to\CH(-)$. By~\cite[Prop.~2.7]{VY} they are
surjective with kernels consisting of nilpotents. The kernel of the
bottom arrow consists of nilpotents by~\cite[Prop~3.1]{VZ}. Then the
kernel of the upper arrow consists of nilpotents as well.

Tensoring the upper arrow with $\Lambda$ we obtain
$\ker(p^{\Omega})\otimes\Lambda\to\Omega_N(X\times X)\otimes\Lambda\stackrel{p^{\Omega}\otimes id}\longrightarrow\Omega(\overline{X}\times\overline{X})\otimes\Lambda$,
so $\ker(p^{\Omega})\otimes\Lambda$ covers the kernel of $p^{\Omega}\otimes id$, thus $\ker(p^{\Omega}\otimes id)$ consists of nilpotents. Now the specialization maps fit into the commutative diagram.
\[
\begin{xymatrix}{
\Omega_N(X\times X)\otimes_{\mathbb{L}}\Lambda\ar[d]^{\nu_{X\times X}}\ar[rr]^{p^{\Omega}\otimes id} && \Omega(\overline{X}\times \overline{X})\otimes_{\mathbb{L}}\Lambda\ar[d]^{\cong} \\
\hh(X\times X)\ar[rr]^{p^{\hh}} && \hh(\overline{X}\times\overline{X})
}\end{xymatrix}
\]
where the right arrow is an isomorphism by~\ref{Cohomology_of_cellular_variety}
and~\ref{structure}, and the map $\nu_{X\times X}$ is surjective. So the kernel of the bottom map consists of nilpotents.
\end{proof}

\medskip

\begin{lem}~\label{filtr_comp} We have
$\hh^{(N+i)}(\overline{X}\times\overline{X})\circ \hh^{(N+j)}(\overline{X}\times\overline{X})\subseteq\hh^{(N+i+j)}(\overline{X}\times\overline{X})$.
\end{lem}
\begin{proof}
Consider a generator $\zeta_m\otimes\tau_n\in\hh^{(N+i)}(\overline{X}\times\overline{X})$ where $N-\alpha_m+\alpha_n\geqslant N+i$
and $\zeta_{m'}\otimes\tau_{n'}\in\hh^{(N+j)}(\overline{X}\times\overline{X})$ where $N-\alpha_{m'}+\alpha_{n'}\geqslant N+j$.
The composition 
\[
(\zeta_{m}\otimes\tau_{n})\circ(\zeta_{m'}\otimes\tau_{n'})=\deg(\tau_{n}\zeta_{m'})(\zeta_{m}\otimes\tau_{n'})=\delta_{n,m'}\cdot(\zeta_{m}\otimes\tau_{n'})
\]
is nonzero iff $n=m'$.
In this case $N-m+n'=(N-m+n)+(N-m'+n')-N\geqslant N+i+j$. Thus $\zeta_{m}\otimes\tau_{n''}$ lies in $\hh^{(N+i+j)}(\overline{X}\times\overline{X}).$
\end{proof}
\begin{rem}
Indeed, the lemma implies that $\hh^{(N)}(\overline{X}\times\overline{X})$ is a ring with respect to the composition product, and $\hh^{(N+1)}(\overline{X}\times\overline{X})$ is its two-sided ideal. Since the composition of homogeneous elements is homogeneous, $\hh^{(N)}_N(\overline{X}\times\overline{X})$ is also a ring with respect to the composition.
\end{rem}
\begin{lem}
The isomorphism $\Phi_N\colon\hh^{(N/N+1)}(\overline{X}\times\overline{X})\to\CH^N(\overline{X}\times\overline{X},\Lambda)$
is a ring homomorphism with respect to the composition product.
\end{lem}
\begin{proof}
This immediately follows from the fact that $\Phi$ maps residue classes of $\zeta_w^{\hh}\otimes\tau_v^{\hh}$ to $\zeta_w^{\tCH}\otimes\tau_v^{\tCH}$.
\end{proof}

\begin{lem}\label{rat}
Let $Y$ be a twisted form of $X$, i.e. $Y_L\cong X_L=\overline{X}$. For every codimension $m$ consider the diagram, where $p\colon\overline{X}\times\overline{X}\to X\times Y$ denotes the projection.
\[
\xymatrix{
\hh^{(m)}_{2N-m}(X\times Y)\ar[rr]^{pr_m\circ p^{\hh}} && \hh^{(m/m+1)}_{2N-m}(\overline{X}\times\overline{X})\\
\CH^m(X\times Y,\Lambda^0)\ar[rr]^{p^{\tCH}} && \CH^m(\overline{X}\times\overline{X},\Lambda^0)\ar[u]_{\Phi^m}
}
\]
Then $\im(\Phi^m\circ p^{\tCH})\subseteq\im pr_m\circ p^{\hh}$.
\end{lem}
\begin{proof}
Note that $\CH^m(X\times Y,\Lambda^0)$ is generated over $\Lambda^0$ by classes $i_{\tCH}(1)$ where 
$i\colon\widetilde{Z}\to Z\hookrightarrow X\times Y$, $Z$ is a closed integral subscheme of codimension $m$, $\widetilde{Z}\in\Sm_k$ and $\widetilde{Z}\to Z$ is projective birational.
Consider the Cartesian diagram
\[
\xymatrix{
\widetilde{Z}\ar[d]^{i} && \widetilde{Z}_L\ar[d]^{j}\ar[ll]_{q}\\
X\times_k Y && \overline{X}\times_L\overline{X}\ar[ll]_{p}
}
\]
Since this diagram is transverse, then
\[
j_{\hh}\circ q^{\hh}=p^{\hh}\circ i_{\hh} \text{ and } j_{\tCH}\circ q^{\tCH}=p^{\tCH}\circ i_{\tCH}.
\]
By lemma~\ref{subset} we have
$\Phi^m\circ j_{\tCH}(1)=pr_m(j_{\hh}(1))$.
Then
$
\Phi^m\circ p^{\tCH}(i_{\tCH} (1))=\Phi^m\circ
j_{\tCH}(1)=pr_m(j_{\hh}(1))=pr_m(p^{\hh}\circ i_{\hh}(1))\in \im
pr_m\circ p^{\hh}_m$.
\end{proof}

\begin{lem}\label{subset}
Consider a morphism
$j\colon\widetilde{Z}\to\overline{X}\times_L\overline{X}$, where
$\widetilde{Z}$ is a smooth irreducible scheme and $j$ is projective of relative dimension $-m$.
It induces two pushforward maps $j_{\hh}\colon\hh(\widetilde{Z})\to\hh(\overline{X}\times\overline{X})$ and $j_{\tCH}\colon \CH(\widetilde{Z},\Lambda)\to \CH(\overline{X}\times\overline{X},\Lambda)$.

Then $j_{\hh}(1)\in\hh^{(m)}_{2N-m}(\overline{X}\times\overline{X})$
and $\Phi^m(j_{\tCH}(1))=pr_m(j_{\hh}(1))$.
\end{lem}
\begin{proof}
Observe that 
\[
j_{\hh}(1)=j_{\Omega}(1)\otimes_{\mathbb{L}} 1_{\Lambda} \text{  and   } j_{\tCH}(1)=j_{\Omega}(1)\otimes_{\mathbb{L}} 1_{\mathbb{Z}}.
\] 
Expanding in the basis we obtain
\[
j_{\Omega}(1)=\sum_{i_1,i_2}r_{i_1,i_2}\tau_{i_1}^{\Omega}\otimes\tau_{i_2}^{\Omega}\text{ for some }r_{i_1,i_2}\in\mathbb{L}.\eqno{(*)}
\]
Since $j_{\Omega}(1)$ is homogeneous of degree $m$, we have \[r_{i_1,i_2}\in\mathbb{L}^{m-\alpha_{i_1}-\alpha_{i_2}}.\eqno{(**)}\]
Then for every nonzero $r_{i_1,i_2}$ we have
$\alpha_{i_1}+\alpha_{i_2}\geqslant m$. So each
$\tau_{i_1}^{\Omega}\otimes\tau_{i_2}^{\Omega}\in\Omega^{(m)}(\overline{X}\times\overline{X})$
and, thus,
$j_{\Omega}(1)\in\Omega^{(m)}_{2N-m}(\overline{X}\times\overline{X})$.
Taking $(*)$ modulo $\Omega^{(m+1)}(\overline{X}\times\overline{X})$ we obtain
\[
j_{\Omega}(1)+\Omega^{(m+1)}(\overline{X}\times\overline{X})=\sum_{\alpha_{i_1}+\alpha_{i_2}=m}r_{i_1,i_2}\tau_{i_1}^{\Omega}\otimes\tau_{i_2}^{\Omega}+\Omega^{(m+1)}(\overline{X}\times\overline{X}).
\]
If $\alpha_{i_1}+\alpha_{i_2}=m$ then $r_{i_1,i_2}\in\mathbb{L}^0=\mathbb{Z}$ by $(**)$.
Thus taking $j_{\hh}(1)=j_{\Omega}(1)\otimes_{\mathbb{L}}1_{\Lambda}$ and $j_{\tCH}(1)=j_{\Omega}(1)\otimes_{\mathbb{L}}1_{\mathbb{Z}}$ we get
\[
pr_m(j_{\hh}(1))=pr_m(\sum_{\alpha_{i_1}+\alpha_{i_2}=m}r_{i_1,i_2}\tau_{i_1}^{\hh}\otimes\tau_{i_2}^{\hh})
\]
and 
\[
j_{\tCH}(1)=j_{\Omega}(1)\otimes_{\mathbb{L}} 1_{\mathbb{Z}}\otimes_{\mathbb{Z}}1_{\Lambda}=\sum_{\alpha_{i_1}+\alpha_{i_2}=m}r_{i_1,i_2}\tau_{i_1}^{\tCH}\otimes\tau_{i_2}^{\tCH}.
\]
Then $\Phi^m(j_{\tCH}(1))=pr_m(j_{\hh}(1))$, since $\Phi^m(\tau^{\tCH}_{i_1}\otimes\tau^{\tCH}_{i_2})=pr_m(\tau^{\hh}_{i_1}\otimes\tau^{\hh}_{i_2})$.
\end{proof}

\begin{lem}\label{nilpotent_kernel}
The kernel of the composition homomorphism
\[
pr_N\circ p^{\hh}\colon\hh^{(N)}_{N}(X\times X)\to\hh^{(N)}_{N}(\overline{X}\times\overline{X})\to \hh^{(N/N+1)}_N(\overline{X}\times\overline{X})
\]
consists of nilpotents.
\end{lem}
\begin{proof}
This follows from Rost nilpotence and the fact that $\hh^{(N+1)}(\overline{X}\times\overline{X})$ is nilpotent by Lemma~\ref{filtr_comp}.
\end{proof}

\begin{lem}\label{isomorphic_objects}
Let $\mathcal{C}$ be an additive category, $A,B\in Ob(\mathcal{C})$. Let $f\in Hom_{\mathcal{C}}(A,B)$ and $g\in Hom_{\mathcal{C}}(B,A)$such that $f\circ g-id_B$ is nilpotent in the ring $End_{\mathcal{C}}(B)$ and $g\circ f-id_A$ is nilpotent in the ring $End_{\mathcal{C}}(A).$ Then $A$ is isomorphic to $B$.
\end{lem}
\begin{proof}
Denote $\alpha=id_A-gf$ and $\beta=id_B-fg$. Take natural $n$ such that $\alpha^{n+1}=0$ and $\beta^{n+1}=0.$
Then $gf=id_A-\alpha$ is invertible and $(gf)^{-1}=id_A+\alpha+\ldots+\alpha^n$. Analogously $(fg)^{-1}=id_B+\beta+\ldots+\beta^n$.
So we have 
\[
gf(id_A+(id_A-gf)+\ldots +(id_A-gf)^n)=id_A
\]
Since $(id_A-gf)^n=\sum_{i=0}^n(-1)^i{n \choose i}(gf)^i$, we have
\[
g\sum_{m=0}^n\left(\sum_{i=0}^m(-1)^i{m \choose i}(fg)^{i}f\right)=id_A\eqno{(*)}
\]
and 
\[
fg(id_B+(id_B-fg)+\ldots +(id_B-fg)^n)=id_B
\]
implies
\[
\sum_{m=0}^n\left(\sum_{i=0}^m(-1)^i{m \choose i}f(gf)^i\right)g=id_B.\eqno{(**)}
\]
Then take 
\[
f_1=\sum_{m=0}^n\left(\sum_{i=0}^m(-1)^i{m \choose i}(fg)^if\right)=\sum_{m=0}^n\left(\sum_{i=0}^m(-1)^i{m \choose i}f(gf)^i\right).
\]
Then $(*)$ implies $gf_1=id_A$ and $(**)$ implies $f_1g=id_B.$
So $f_1$ and $g$ establish inverse isomorphisms between $A$ and $B$.
\end{proof}

\begin{cor}\label{two_projectors_cor}
Suppose $p_1$ and $p_2$ are two idempotents in $End(M^{\hh}(\overline{X}))$ such that $p_1-p_2$ is nilpotent.
Then the motives $(\overline{X},p_1)$ and $(\overline{X},p_2)$ are isomorphic.
\end{cor}
\begin{proof}
Take \[
f=p_2\circ p_1\in
Hom_{\Ms_{\hh}}((\overline{X},p_1),(\overline{X},p_2))\text{ and }g=p_1\circ p_2\in Hom_{\Ms_{\hh}}((\overline{X},p_2),(\overline{X},p_1)).\] 
Let us check that $f\circ
g-id_{(X,p_2)}=p_2p_1p_2-p_2=p_2(p_1-p_2)p_2$ is nilpotent. 

It is sufficient to check that $(p_2(p_1-p_2)p_2)^m=p_2(p_1-p_2)^mp_2$ for any $m$.
Note that if $x\in\ker p_2\cap \im p_1$ then $(p_1-p_2)(x)=x$. Since $p_1-p_2$ is nilpotent, $x=0$. Thus, $\ker p_2\cap\im p_1=0$.
Since $p_2$ is idempotent, $\im p_2\cap\ker p_2=0$. Then endomorphism $p_1-p_2$ of $M(\overline{X})=\ker p_2\oplus\im p_2$ can be represented as the matrix
\[
p_1-p_2=\begin{pmatrix}
E_1 & E_2\\
0 & 0
\end{pmatrix}
\]
where $E_1$ is a homomorphism from $\im p_2$ to $\im p_2$ and $E_2$ is
a homomorphism from $\ker p_2$ to $\im p_2$.
We have
\[p_2(p_1-p_2)^mp_2=p_2\circ\begin{pmatrix}
E_1^m & E_1^{m-1}E_2\\
0 & 0
\end{pmatrix}\circ p_2=\begin{pmatrix}
E_1^m & 0\\
0 & 0
\end{pmatrix}=(p_2(p_1-p_2)p_2)^m.
\]
Then $f\circ g-id_{(X,p_2)}=p_2(p_1-p_2)p_2$ is
nilpotent. Symmetrically, $g\circ f-id_{(X,p_1)}$ is nilpotent. 
So $(X,p_1)$ and $(X,p_2)$ are isomorphic by~Lemma~\ref{isomorphic_objects}.
\end{proof}

We are now ready to prove theorems A, B and C of the introduction:

\begin{mthmA}\label{main_decomp}
Suppose $X$ is generically cellular.
Assume that there is a decomposition of Chow motive with coefficients in $\Lambda^0$
\[M^{\tCH}(X,\Lambda^0)=\bigoplus_{i=0}^n \Rs(\alpha_i)\eqno{(*)}\] such that over the splitting field $L$ the motive $\Rs$ equals to the sum of twisted Tate motives:
$\overline{\Rs}=\bigoplus_{j=0}^m\Lambda^0(\beta_j)$.

Then there is a $\hh$-motive $\Rs_{\hh}$ such that 
\[
M^{\hh}(X)=\bigoplus_{i=0}^{n} \Rs_{\hh}(\alpha_i)
\]
such that over the splitting field $\Rs_{\hh}$ splits into the $\hh$-Tate motives $\overline{\Rs_{\hh}}=\bigoplus_{j=0}^m\Lambda(\beta_j)$.
\end{mthmA}

\begin{proof}
We may assume that $\alpha_0=0$ in ($*$). Then each summand $\Rs(\alpha_i)$ equals to $(X,p_i)$ for some idempotent $p_i$
and there are mutually inverse isomorphisms $\phi_i$ and $\psi_i$ of degree $\alpha_i$ between $(X,p_0)$ and $(X,p_i)$.
So we have
\begin{itemize}
\item idempotents $p_{i}\in\CH^N(X\times X),$ $\sum p_i=\Delta^{X}_{\hh}(1)$
\item isomorphisms $\phi_{i}\in p_{0}\circ\CH^{N+\alpha_i}(X\times X)\circ p_{i}$ and 
$\psi_{i}\in p_{i}\circ\CH^{N-\alpha_i}(X\times X)\circ p_{0}$ 
\item such that $\phi_{i}\circ\psi_{i}=p_{0}$ and $\psi_{i}\circ\phi_{i}=p_{i}$
\end{itemize}
Consider the diagram of Lemma~\ref{rat}
\[
\xymatrix{
\hh^{(m)}_{2N-m}(X\times Y)\ar[rr]^{pr_m\circ p^{\hh}} && \hh^{(m/m+1)}_{2N-m}(\overline{X}\times\overline{X})\\
\CH^m(X\times Y,\Lambda^0)\ar[rr]^{p^{\tCH}} && \CH^m(\overline{X}\times\overline{X},\Lambda^0)\ar[u]_{\Phi^m}
}.
\]

By~\ref{rat} the elements $\Phi^N\circ p^{\tCH}(p_{i})$ and $\Phi^{N+\alpha_i}\circ p^{\tCH}(\phi_{i})$ and $\Phi^{N-\alpha_i}\circ p^{\tCH}(\psi_{i})$ lie in $\im pr_{N}\circ p^{\hh}$, $\im pr_{N-\alpha_i}\circ p^{\hh}$ and $\im pr_{N+\alpha_i}\circ p^{\hh}$ respectively.

By~Lemma~\ref{nilpotent_kernel} the kernel of $pr_N\circ p^{\hh}\colon\hh^{(N)}_N(X\times X)\to\hh^{(N/N+1)}_N(\overline{X}\times\overline{X})$ is nilpotent. Then by~\cite[Prop.~27.4]{AF} there is a decomposition $r_{i}$ such that $pr_N\circ p^{\hh}(r_{i})=p_{i}$. 

Let us construct the isomorphisms between $r_{i}$ and $r_{0}$.
Let $\phi'_{i}$ and $\psi'_{i}$ be some preimages of $\Phi^{N+i}\circ
p^{\tCH}(\phi_{i})$ and $\Phi^{N-i}\circ p^{\tCH}(\psi_{i})$. 
Then \cite[Lem.~2.5]{PSZ} implies that there are elements 
$\phi''_{i}\in r_{0}\hh^{(N)}_N(X\times X)r_{i}\text{ and }\psi''_{i}\in r_{i}\hh^{(N)}_N(X\times X) r_{0}$, 
such that $\phi_{i,j}\psi_{i,j}=r_{0,1}$ and $\psi_{i,j}\phi_{i,j}=r_{i,j}$. So the $\hh$-motives $(X,r_{i})$ and $(X,r_{0})(\alpha_i)$ are isomorphic. Taking $\Rs_{\hh}=(X,r_{0})$ we have
\[
M^{\hh}(X)=\bigoplus_{i=0}^n(X,r_{i})=\bigoplus_{i=0}^n(X,r_{0})(\alpha_i)=\bigoplus_{i=0}^n \Rs_{\hh}(\alpha_i).
\]

Over the splitting field the motive $\overline{\Rs_{\hh}}$ becomes isomorphic to $(\overline{X},p^{\hh}(r_{0}))$
and $pr_N\circ p^{\hh}(r_{0})=\Phi^{N}(p^{\tCH}(p_{0})).$
Since the Chow motive $(\overline{X},p^{\tCH}(p_{0,1}))$ splits into $\bigoplus_{j}\Lambda^0(\beta_j)$ we have
$p^{\tCH}(p_{0,1})=\sum_{j}f_{j}\otimes g_{j}$ with
$f_{j}\in\CH^{\alpha_j}(\overline{X})$,
$g_{j}\in\CH_{\alpha_j}(\overline{X})$ and
$\pi_{\tCH}(f_{j}g_{l})=\delta_{j,l}$. Take $\varphi_{j}$ and
$\gamma_j$ to be the liftings of $f_j$ and $g_j$ in $\hh^{(\alpha_j)}_{N-\alpha_j}(\overline{X})$ and $\hh^{(N-\alpha_j)}_j(\overline{X})$ respectively. 

Note that
$\varphi_{j}\gamma_{l}+\hh^{N+1}(\overline{X})=\Psi^N(f_{j}g_{l})$. Since
$\hh^{(N+1)}(\overline{X})=0$, we have
$\pi_{\hh}(\varphi_{j}\gamma_{l})=\pi_{\tCH}(f_{j}g_{l})=\delta_{j,l}$. Then
the element $\sum\varphi_{j}\otimes\gamma_{j}$ is an idempotent in $Corr_0(\overline{X}\times\overline{X})$. Since \[
pr_N(p^{\hh}(r_0))=\Phi^N(p^{\tCH}(p_{0}))=pr_N(\sum_{j}\varphi_{j}\otimes\gamma_{j}), 
\]
$p^{\hh}(r_{0})-\sum\varphi_{j}\otimes\gamma_{j}$ lies in $\hh^{N+1}(\overline{X}\times\overline{X})$, so is nilpotent. Then by Corollary~\ref{two_projectors_cor} we obtain
\[
\overline{\Rs_{\hh}}=(\overline{X},p^{\hh}(r_{0}))\cong(\overline{X},\sum_{j}\varphi_{j}\otimes\gamma_{j})=\bigoplus\Lambda(\beta_j).\qedhere
\]
\end{proof}

\begin{lem}\label{indecomp_lem} 
Assume that $\Lambda^1=\ldots\Lambda^N=0$. Then $\hh_N(\overline{X}\times\overline{X})\subseteq\hh^{(N)}(\overline{X}\times\overline{X})$ and in the diagram of Lemma~\ref{rat}
\[ 
\xymatrix{
\hh^{(N)}_{N}(X\times Y)\ar[rr]^{pr_N\circ p^{\hh}} && \hh^{(N/N+1)}_{N}(\overline{X}\times\overline{X})\\
\CH^N(X\times Y,\Lambda^0)\ar[rr]^{p^{\tCH}} && \CH^N\overline{X}\times\overline{X},\Lambda^0)\ar[u]_{\Phi^N}
}
\]
the inverse inclusion holds: $\im pr_N\circ p^{\hh}\subseteq \im \Phi^N\circ p^{\tCH}$.
\end{lem}

\begin{proof}
By the degree formula~\cite[Thm 4.4.7]{LM} $\hh(X\times X)$ is
generated as $\Lambda$-module by pushforwards $i_{\hh}(1)$, where
$i\colon Z\to X\times X$ is projective, $Z\in Sm_k$ and $i\colon Z\to
i(Z)$ is birational. Following~\cite{LM} we will denote such classes
by $[Z\to X\times X]_{\hh}$. Then $\hh_N(X\times X)$ is additively
generated by elements $\lambda[Z\to X\times X]_{\hh}$, where $\lambda$
is homogeneous such that $\deg \lambda+\codim Z=N$. Since
$\Lambda^1=\ldots\Lambda^N=0$, we have  $\codim Z\geqslant N$. Then in
$\Omega(\overline{X}\times\overline{X})$ we have 
\[
[Z_L\to\overline{X}\times\overline{X}]_{\Omega}=\sum
\omega_{i,j}\zeta_i\otimes\tau_j\text{ for some } \omega_{i,j}\in\mathbb{L}
.\] 
Since all elements of the Lazard ring have negative degrees and
$[Z_L\to\overline{X}\times\overline{X}]_{\Omega}$ has degree $N$, each
$\zeta_i\otimes\tau_j$ in the expansion is contained in
$\Omega^{(n)}(\overline{X}\times\overline{X})$. 
Then
\[
[Z_L\to\overline{X}\times\overline{X}]_{\hh}=\nu_{\overline{X}\times\overline{X}}[Z_L\to\overline{X}\times\overline{X}]_{\Omega}
\in \hh^{(N)}(\overline{X}\times\overline{X})\text{ and }\]
$[Z\to X\times X]_{\hh} \in \hh^{(N)}(X\times X)$. By the same reasons
$[Y\to X\times X]$ belongs to $\hh^{(N+1)}(X\times X)$ if $\codim
Y>N$. Then $\im pr_N\circ p^{\hh}$ is generated over $\Lambda^0$ by
classes of $[Z_L\to\overline{X}\times\overline{X}]_{\hh}$, where $Z\to
X\times X$ has codimension $N$.

By Lemma~\ref{subset} for any $Z\to X\times X$ of codimension $N$ we
have \[
pr_N\circ p^{\hh}([Z\to X\times X]_{\hh})=\Phi^N\circ p^{\tCH}([Z\to
X\times X]).\] 
Then $\im pr_N\circ p^{\hh}\subseteq \Phi^N\circ p^{\tCH}$ and the
theorem is proven.
\end{proof}

\begin{mthmB}\label{Rost_indecomp}
Let $\hh$ be oriented cohomology theory with coefficient ring $\Lambda$. Assume that the Chow motive $\Rs$ is indecomposable over $\Lambda^0$ and $\Lambda^{1}=\ldots=\Lambda^{N}=0$.
Then the $\hh$-motive $\Rs_{\hh}$ from theorem~A is indecomposable.
\end{mthmB}

\begin{proof} 
By definition, $\Rs_{\hh}=(X,r_{0})$ where $r_{0}$ is an idempotent in $\hh^{(N)}_N(X\times X)$.
If $\Rs_{\hh}$ is decomposable, then $r_0=r_1+r_2$ for some idempotents in $r_1,r_2\in\hh_N(X\times X)$
Then by Lemma~\ref{indecomp_lem} $r_1,r_2\in\hh^{(N)}_N(X\times X)$
and $p_1=(\Phi^N)^{-1}\circ pr_N\circ p^{\hh}(r_1)$ and
$p_2=(\Phi^N)^{-1}\circ pr_N\circ p^{\hh}(r_2)$ are rational
idempotents and $p^{\tCH}(p_0)=p_1+p_2$. These idempotents are
nontrivial, since $\ker (\Phi^N)^{-1}\circ pr_N\circ p^{\hh}$ is
nilpotent. Hence, the Chow motive $\Rs=(X,p_{0})$ is decomposable, a contradiction.
\end{proof}

\begin{ex}
If $\hh=\Omega$ or connective $K$-theory,  all the elements in the
coefficient ring have negative degree. 
Then Theorems~A and B prove that
$\hh$-motivic irreducible decomposition coincides with
integral Chow-motivic decomposition. This gives another proof of the result by Vishik-Yagita~\cite[Cor.~2.8]{VY}.
\end{ex}

\begin{ex}
Take $\hh$ to be Morava K-theory $\hh=K(n)^*$. The coefficient ring is
$\mathbb{F}_p[v_n,v_n^{-1}]$, where $\deg(v_n)=-2(p^n-1)$. In the
case $n>\log_p(\frac{N}{2}+1)$ Theorems~A and B prove that $M^{K(n)}(X)$ has the same irreducible decomposition as Chow motive modulo $p$.
\end{ex}

\begin{mthmC}
Suppose that $X,Y$ are generically cellular and $Y$ is a twisted form of $X$, i.e. $\overline{Y}\cong\overline{X}$.

If $M^{\tCH}(X,\Lambda^0)\cong M^{\tCH}(Y,\Lambda^0)$, then
$M^{\hh}(X)\cong M^{\hh}(Y)$.
\end{mthmC}

\begin{proof}
Let $f\in\CH^{N}(X\times Y)$ and $g\in\CH^N(Y\times X)$ be correspondences, that give mutually inverse isomorphisms between $M^{\tCH}(X)$ and $M^{\tCH}(Y)$.
Consider the diagram
\[
\begin{xymatrix}{
\hh^{(N)}_N(X\times Y)\ar[r]^{pr_N\circ p^{\hh}} & \hh^{(N/N+1)}_N(\overline{X}\times\overline{X})\\
\CH^N(X\times Y,\Lambda^0)\ar[r]^{p^{\tCH}} & \CH^N(\overline{X}\times\overline{X},\Lambda^0)\ar[u]_{\Phi^N}}
\end{xymatrix}
\]
Then by Lemma~\ref{rat} we can find $f_1\in\hh^{(N)}_N(X\times Y)$ and $g_1\in\hh^{(N)}_N(Y\times X)$ such that $pr_N\circ p^{\hh}(f_1)=\Phi^N(f)$ and $pr_N\circ p^{\hh}(g_1)=\Phi^N(g)$. Then $g_1\circ f_1-\Delta_{X}$ lies in the kernel of the map
\[
\hh^{(N)}_N(X\times X)\stackrel{pr_N\circ p^{\hh}}\longrightarrow\hh^{(N/N+1)}_N(\overline{X}\times\overline{X})
\]
which consists of nilpotents by Lemma~\ref{nilpotent_kernel}. So
$g_1\circ f_1-\Delta_{X}$ is nilpotent. 
By the same reasons $f_1\circ g_1-\Delta_{Y}$ is nilpotent. Then $M^{\hh}(X)$ and $M^{\hh}(Y)$ are isomorphic by Lemma~\ref{isomorphic_objects} and the theorem is proven.
\end{proof}

\bibliographystyle{plain}

\end{document}